\documentclass[a4]{amsart}

\input xypic 
\input xy 
\xyoption{all} 
\usepackage{amssymb} 
\usepackage{bbm}
\usepackage{tikz-cd}

\usepackage{float}
\usepackage{graphicx}
\usepackage{accents}
\usepackage{mathrsfs}  

\usepackage{faktor}

\usepackage[bookmarksnumbered, bookmarksopen, backref]{hyperref}
\newtheorem{theorem}{Theorem}[section]

\newtheorem{lemma}[theorem]{Lemma}
\newtheorem{corollary}[theorem]{Corollary}

\theoremstyle{definition}

\newtheorem{remark}[theorem]{Remark}

\newtheorem{question}[theorem]{Question}

\newcounter{RomanNumber}

\newcommand{\Zpr}{\mathbb{Z}/p^r}
\newcommand{\Zps}{\mathbb{Z}/p^s}

\newcommand{\Zpt}{\mathbb{Z}/p^t}
\newcommand{\Zp}{\mathbb{Z}/p}

\newcounter{bean}

\newcommand{\qqed}{\hfill\Box}

\begin{document}
\begin{sloppypar}
%%% Title

\title{Some asymptotic formulae for torsion in homotopy groups} 
%author two information
\author{Guy Boyde}
\address{Mathematical Institute, Utrecht University, Heidelberglaan 8
3584 CS Utrecht, The Netherlands}
\email{g.boyde@uu.nl}

\author{Ruizhi Huang} 
\address{Institute of Mathematics, Academy of Mathematics and Systems Science, 
 Chinese Academy of Sciences, Beijing 100190, China} 
\email{haungrz@amss.ac.cn} 
   \urladdr{https://sites.google.com/site/hrzsea/}
   \thanks{}

%    \subjclass is required.
\subjclass[2010]{Primary 
55Q52, %Homotopy groups of special spaces
55Q05, %Homotopy groups, general; sets of homotopy classes
%55P35, %Loop spaces
%55R25, %Sphere bundles and vector bundles in algebraic topology
%57R22, %Topology of vector bundles and fiber bundles
%14F35;  %Homotopy theory and fundamental groups in algebraic geometry
%57R19,  %Algebraic topology on manifolds and differential topology
Secondary 
55Q15, %Whitehead products and generalizations
%57R65,  %Surgery and handlebodies
%55Q50,  %$J$-morphism 
55P40.  %Suspensions
%55P62. %rational homotopy theory
}
\keywords{}
\date{}
%\thanks{}

%%% Abstract

\begin{abstract} 
Inspired by a remarkable work of F\'{e}lix, Halperin and Thomas on the asymptotic estimation of the ranks of rational homotopy groups, and more recent works of Wu and the authors on local hyperbolicity, we prove two asymptotic formulae for torsion rank of homotopy groups, one using ordinary homology and one using $K$-theory. We use these to obtain explicit quantitative asymptotic lower bounds on the torsion rank of the homotopy groups for many interesting spaces after suspension, including Moore spaces, Eilenberg-MacLane spaces, complex projective spaces, complex Grassmannians, Milnor hypersurfaces and unitary groups.
\end{abstract}

\maketitle
%\tableofcontents
%%%%%%%%%%%%%%%%%%%%%%%%%%%%%%%%%%%%%%%%%%%%%%%%%%%%%%%%%%%%%%%%%%%%%%%

\section{Introduction}
The homotopy groups of a simply connected $CW$-complex $Y$ of finite type have the form
\[
\pi_i(Y)\cong (\mathop{\oplus}_{d_i}\mathbb{Z})\oplus \mathop{\oplus}\limits_{\substack{{\rm prime}~p\\t \in \mathbb{Z}^+}}(\mathop{\oplus}\limits_{k_{p,t}} \mathbb{Z}/p^t),
\]
where $d_i$ and $k_{p,t}$ are the {\it rank} of the free summands and the $\mathbb{Z}/p^t$-summands of $\pi_i(Y)$ respectively. Denote ${\rm rank}_{0}(\pi_i(Y)):=d_i$ and ${\rm rank}_{\mathbb{Z}/p^t}(\pi_i(Y)):=k_{p,t}$. 

In the remarkable work \cite{FHT}, F\'{e}lix, Halperin and Thomas proved an asymptotic formula for the ranks ${\rm rank}_{0}(\pi_i(Y))$ of the free part of the homotopy groups. In particular, they showed that if $Y$ is finite and ${\rm rank}_{0}(\pi_i(Y))\neq 0$ for infinitely many $i\in \mathbb{Z}^{+}$, then there is a constant $\delta>1$ such that for $N$ large enough
\[
\sum_{i=N+2}^{N+{\rm dim}(Y)} {\rm rank}_{0}(\pi_i(Y))\geq \delta^{N},
\]
which they interpret as a strong `regularity' property for the ranks ${\rm rank}_{0}(\pi_i(Y))$ of the free part of the homotopy groups.  
Concerning the ranks ${\rm rank}_{\mathbb{Z}/p^t}(\pi_i(Y))$ of the torsion part of the homotopy groups, they further raised the following natural question, which was rephrased in an explicit form by Wu and the second author in \cite[Question 1.8]{HW}.
\begin{question}
Are there `regularity' properties of the torsion subgroups of the homotopy groups $\pi_i(Y)$ as $i\rightarrow \infty$?  
\end{question}

In this paper, we study the above question by providing estimates for the ranks ${\rm rank}_{\mathbb{Z}/p^t}(\pi_i(Y))$ in certain cases. In particular, we give quantitative refinements of results of the authors and Wu, from the papers \cite{HW}, \cite{Boy1}, and \cite{Boy2}. The methods of these papers implied more than was stated in the theorems: the statements were always that the volume of $p$-torsion in the homotopy groups of various spaces grows exponentially, but actually the methods were completely constructive, and with more work one can extract concrete exponential lower bounds. The extraction of these lower bounds is the business of this paper.

The proof of each of our main theorems (\ref{homologyResultNoSuspension} and \ref{secondKTheoremw}) begins with the combinatorics of free Lie algebras, which have been well understood since long before F\'elix, Halperin, and Thomas's celebrated theorem. We use in particular some results of Babenko \cite{Bab} and Lambrechts \cite{Lam}, both of which are more general. From this common beginning, the proof of each theorem is then complicated in a different way; we elaborate briefly after each theorem statement.

Recent work of Burklund and Senger \cite{BS} has greatly advanced our understanding of these phenomena: they finish a story begun by Henn \cite{Hen} and Iriye \cite{Iri} and show that the radii of convergence of the $p$-local ``homotopy'' and ``loop-homology'' power series are equal. Again, we discuss each of our theorems in light of this.

\subsection{Results via homology} We first give our quantitative refinement of the main result of \cite{Boy2}. 
To state the results, for any integer $q\geq 2$ define a function 
\[
f_q(x)=(1 - \frac{x}{x-1} \frac{1}{\varphi}) \cdot \frac{1}{x} \varphi^{x} - c x \varphi^{\frac{x}{2}} -  \kappa \lvert \psi \rvert^x
\]
for $x\geq 2$, where 
\begin{itemize}
\item $\varphi $ is the unique positive real root of the degree $q+1$ polynomial $P(z)=z^{q+1} - z - 1$,
\item $c= 2(q+2)(1+ \varphi)$, and
\item $\kappa = (q+1)(1+ \frac{1}{\lvert \psi \rvert})$ with $\psi$ the next largest root of $P(z)$ in absolute value.
\end{itemize}
We have the properties
\begin{itemize}
\item $2^{\frac{1}{q+1}}< \varphi < 1 + \frac{1}{q}$, and
\item for any $\varepsilon >0$, once $x$ is large enough we have 
\[
f_q(x) \geq (1- \varepsilon) (1 - \frac{1}{\varphi}) \cdot \frac{1}{x} \varphi^{x} > (1- \varepsilon) (1 - 2^{-\frac{1}{q+1}}) \cdot \frac{1}{x} 2^{\frac{x}{q+1}}.
\]
\end{itemize}
We will use the function $f_q(x)$ with its properties freely in this subsection.

Let $P^{q+1}(p^r)$ be the Moore space defined as the mapping cone of the degree $p^r$ map $S^{q}\rightarrow S^q$.
The following theorem provides an asymptotic formula for the $p$-local homotopy groups under a homological condition. 
\begin{theorem} \label{homologyResultNoSuspension} 
Let $Y$ be a simply connected $CW$-complex, let $p \neq 2$ be prime, and let $s \leq r \in \mathbb{Z}^+$. If there exists a map 
$$\mu: P^{q+1}(p^r) \longrightarrow Y$$ 
for some $q\geq 2$, 
such that the induced map $$(\Omega \mu)_* : H_*(\Omega P^{q+1}(p^r);\Zps) \longrightarrow H_*(\Omega Y;\Zps)$$ is an injection, then we have the bound $$\sum_{t=s}^r {\rm rank}_{\Zpt}(\pi_{N+1}(Y)) \geq f_q(N).$$ 
In particular, for any $\varepsilon >0$, once $N$ is large enough we have 
\[
\sum_{t=s}^r {\rm rank}_{\Zpt}(\pi_{N+1}(Y)) >(1- \varepsilon) (1 - 2^{-\frac{1}{q+1}}) \cdot \frac{1}{N} 2^{\frac{N}{q+1}}.
\]
\end{theorem}
Algebraically, Theorem \ref{homologyResultNoSuspension} depends on the structure of the module of boundaries in a free Lie algebra over a finite field. It is the need to take boundaries which complicates the story relative to Babenko and Lambrechts's work. This is dealt with in Subsection \ref{diffSubsection}, using a result of Cohen, Moore, and Neisendorfer \cite{CMN}.

It follows from \cite{Boy2} that the hypotheses of Theorem \ref{homologyResultNoSuspension} simplify in the case that $Y=\Sigma X$ is a suspension, as follows:

\begin{theorem} \label{homologyResultSuspension} 
Let $X$ be a connected $CW$-complex, let $p \neq 2$ be prime, and let $s \leq r \in \mathbb{Z}^+$. Suppose that $H_*(X; \Zps)$ has finite type. If there exists a map 
$$\mu: P^{q+1}(p^r) \longrightarrow \Sigma X$$ 
for some $q\geq 2$, 
such that $$\mu_* : \widetilde{H}_*(P^{q+1}(p^r);\Zps) \longrightarrow \widetilde{H}_*(\Sigma X;\Zps)$$ is an injection, then we have the bound $$\sum_{t=s}^r {\rm rank}_{\Zpt}(\pi_{N+1}(\Sigma X))  \geq f_q(N).$$ 
In particular, for any $\varepsilon >0$, once $N$ is large enough we have 
\[
\sum_{t=s}^r {\rm rank}_{\Zpt}(\pi_{N+1}(\Sigma X)) > (1- \varepsilon) (1 - 2^{-\frac{1}{q+1}}) \cdot \frac{1}{N} 2^{\frac{N}{q+1}}.
\] 
\end{theorem}
The spaces $X$ and $Y$ in Theorems \ref{homologyResultNoSuspension} and \ref{homologyResultSuspension} can be infinite. 
The asymptotic formulae in both theorems bound the ranks of the $p$-local homotopy groups from below by an exponential function. In particular, they strengthen a recent result of the first author on local hyperbolicity \cite[Theorem 1.5 and 1.6]{Boy2}.

Theorem \ref{homologyResultSuspension} has interesting applications. For instance, we can show the following corollary.
\begin{corollary}\label{coro2intro}
Let $X$ be a $(q-2)$-connected $CW$-complex with $q\geq 2$, and let $p \neq 2$ be prime. Suppose that $H_*(X; \Zpr)$ has finite type and $H_q(\Sigma X; \mathbb{Z})$ contains a $\mathbb{Z}/p^r$-summand. Then we have
\[ 
{\rm rank}_{\Zpr}(\pi_{N+1}(\Sigma X))  \geq f_q(N).
\]
In particular, for any $\varepsilon >0$, once $N$ is large enough we have 
\[\pushQED{\qed} 
{\rm rank}_{\Zpr}(\pi_{N+1}(\Sigma X))  > (1- \varepsilon) (1 - 2^{-\frac{1}{q+1}}) \cdot \frac{1}{N} 2^{\frac{N}{q+1}}.
\qedhere\popQED 
\]
\end{corollary}
Either Corollary \ref{coro2intro} or Theorem \ref{homologyResultSuspension} implies the following immediately for the Moore spaces.
\begin{corollary}\label{coro1intro}
Let $p$ be an odd prime and $q\geq 2$. Then
\[
{\rm rank}_{\mathbb{Z}/p^r}(\pi_{N+1}(P^{q+1}(p^r)) \geq f_q(N).
\]
In particular, for any $\varepsilon >0$, once $N$ is large enough we have 
\[\pushQED{\qed} 
{\rm rank}_{\mathbb{Z}/p^r}(\pi_{N+1}(P^{q+1}(p^r))  > (1- \varepsilon) (1 - 2^{-\frac{1}{q+1}}) \cdot \frac{1}{N} 2^{\frac{N}{q+1}}.
\qedhere\popQED 
\]
 \end{corollary}

This strengthens a result of Wu and the second author on the $\mathbb{Z}/p^r$-hyperbolicity of $P^{q+1}(p^r)$ \cite[Theorem 1.6]{HW}. 

It is enlightening to compare this result from what could be deduced already from Burklund and Senger's work \cite{BS}. It follows from their Corollary A.5 that the radius of convergence of the series $\sum_{N=1}^\infty \dim_{\Zp}(\pi_N(P^{q+1}(p^r)) \otimes \Zp) \cdot t^N$ is precisely $\frac{1}{\varphi}$. Since $P^{q+1}(p^r)$ is rationally elliptic (being rationally contractible), this power series really is describing the torsion. Corollary \ref{coro1intro} adds information in two ways: first by giving a concrete function as a lower bound for all $N$, and second by saying something about summands isomorphic to $\Zpr$ in particular, rather than $p$-torsion in general.

Another interesting example of Corollary \ref{coro2intro} is Eilenberg-MacLane space after suspension. In particular, the following immediate corollary strengthens \cite[Example 2.5]{Boy2}.
\begin{corollary}\label{coro3intro}
Let $p$ be an odd prime and $q\geq 2$. Then
\[ 
{\rm rank}_{\mathbb{Z}/p^r}(\pi_N(\Sigma K(\mathbb{Z}/p^r, q-1))  \geq f_q(N).
\]
In particular, for any $\varepsilon >0$, once $N$ is large enough we have 
\[\pushQED{\qed} 
{\rm rank}_{\mathbb{Z}/p^r}(\pi_N(\Sigma K(\mathbb{Z}/p^r, q-1)) > (1- \varepsilon) (1 - 2^{-\frac{1}{q+1}}) \cdot \frac{1}{N} 2^{\frac{N}{q+1}}.
\qedhere\popQED 
\]

\end{corollary}

\subsection{Results via $K$-theory} 
Denote by
\[
\mathrm{rank}_{p}(\pi_{i}(\Sigma X))=\sum_{t=1}^\infty \mathrm{rank}_{\Zpt}(\pi_{i}(\Sigma X))
\]
the {\it rank} of the $p$-torsion summands of $\pi_{i}(\Sigma X)$.
Our other main result refines the main theorem of \cite{Boy1} to a quantitative statement under a $K$-theoretical condition:

\begin{theorem}[Weak version of Theorem \ref{secondKTheorem}] \label{secondKTheoremw} Let $p$ be an odd prime, and let $X$ be a path connected space having the $p$-local homotopy type of a finite $CW$-complex. Suppose that there exists a map $$\mu : \bigvee_{i=1}^\ell \bigvee_{j=1}^{m_i} S^{q_i+1} \to \Sigma X$$
with $1 \leq q_1 < q_2 < \dots < q_\ell$, such that the map $$\widetilde{K}^*(\Sigma X) \otimes \Zp \xrightarrow{\mu^*} \widetilde{K}^*(\bigvee_{i=1}^\ell \bigvee_{j=1}^{m_i} S^{q_i+1}) \otimes \Zp \cong \bigoplus_{i=1}^\ell \bigoplus_{j=1}^{m_i} \Zp$$
is a surjection.

Then for any $\varepsilon >0$, once the multiple $M=mg'$ of $$g'=\gcd(q_1, \dots q_\ell,2(p-1))$$ 
is large enough we have 
$$\mathrm{rank}_{p}(\pi_{M}(\Sigma X))\geq \frac{1}{M^{1+\varepsilon}}\varphi^{(\frac{\mathrm{conn}(X)+1}{\dim(X)+1})M},$$
where $\varphi$ is the unique positive real root of the degree $q_\ell$ polynomial $$z^{q_\ell} - \sum_{i=1}^\ell m_i z^{q_\ell - q_i} = 0,$$ (in particular, $\varphi \geq (\sum_{i=0}^\ell m_i)^{\frac{1}{q_\ell}} = (\sum_{i=0}^\ell m_i)^{\frac{1}{\max(q_1, \dots q_\ell)}}$), $\mathrm{conn}(X)$ is the $p$-local connectivity of $X$, and $\dim(X)$ is the rational cohomological dimension of $X$. \end{theorem}

%The constants $\tau,\theta$ and the error term can be given explicitly, and we do so in Remark \ref{Constants} at the end of the paper. 
A stronger estimate is provided by Theorem \ref{secondKTheorem} with Remark \ref{Constants} at the end of the paper.
In particular, the asymptotic formulae in both theorems bound the ranks of the $p$-local homotopy groups from below by an exponential function. 
Unlike with Theorem \ref{homologyResultNoSuspension}, it is not necessary to take boundaries to prove Theorem \ref{secondKTheoremw}, but the topological picture is difficult. The difficulty arises ultimately from an interaction between the James construction and the Adams operations, which as far as the authors know originates in the paper \cite{Sel} of Selick on which \cite{Boy1} is modelled, and manifests combinatorially as Condition (\ref{eqn}) in the last section. This condition means that the `Lie algebra' one is ultimately able to find a copy of in homotopy groups `lags' - appearing in higher dimensions than one might expect. In the end, this shows up as for example the factor of $(\frac{\mathrm{conn}(X)+1}{\dim(X)+1})$ in the exponent in Theorem \ref{secondKTheoremw}.

Theorem \ref{secondKTheoremw} has interesting applications. For instance, 
let ${\rm Gr}_k(\mathbb{C}^n)$ be the Grassmannian of $k$-dimensional complex linear subspaces of $\mathbb{C}^n$, which is simply connected and of complex dimension $k(n-k)$. 
Recall ${\rm Gr}_1(\mathbb{C}^n)\cong \mathbb{C}P^{n-1}$.
In \cite[Example 2.6]{Boy1}, it is shown that when $n\geq 3$ and $0<k<n$ there is a map 
\[
S^3\vee S^5 \stackrel{}{\longrightarrow}\Sigma {\rm Gr}_k(\mathbb{C}^n)
\]
which induces a surjection on $\widetilde{K}^\ast(~)\otimes \mathbb{Z}/p$ for all odd primes $p$. Therefore, the following corollary follows immediately from Theorem \ref{secondKTheoremw}, which strengthens \cite[Example 2.5 and 2.6]{Boy1}.
\begin{corollary}\label{Kcoro1intro}
Let $p$ be an odd prime, $n\geq 3$ and $0<k<n$.
Then for any $\varepsilon >0$, once $m$ is large enough we have
\[\pushQED{\qed} 
\mathrm{rank}_{p}(\pi_{2m}(\Sigma {\rm Gr}_k(\mathbb{C}^n)))\geq \frac{1}{(2m)^{1+\varepsilon}}\Big(\frac{3+\sqrt{5}}{2}\Big)^{\frac{m}{2k(n-k)+1}}.
\qedhere\popQED 
\]
\end{corollary}

Similarly, let $H_{n, \ell}$ be the Milnor hypersurface defined by
\[
H_{n, \ell}=\{([z], [w])\in \mathbb{C}P^n\times \mathbb{C}P^\ell~|~ \sum\limits_{i=0}^{{\rm min}(n, \ell)} z_iw_i=0\},
\]
which is simply connected and of complex dimension $n+\ell-1$.
In \cite[Example 2.7]{Boy1}, it is showed that when $n\geq2$ and $\ell\geq 3$ there is a map 
\[
S^3\vee S^5 \stackrel{}{\longrightarrow}\Sigma H_{n, \ell}
\]
which induces a surjection on $\widetilde{K}^\ast(~)\otimes \mathbb{Z}/p$ for all odd primes $p$. 
Therefore, the following corollary follows immediately from Theorem \ref{secondKTheoremw}, which strengthens \cite[Example 2.7]{Boy1}.
\begin{corollary}\label{Kcoro2intro}
Let $p$ be an odd prime, $n\geq2$ and $\ell\geq 3$.
Then for any $\varepsilon >0$, once $m$ is large enough we have
\[\pushQED{\qed} 
\mathrm{rank}_{p}(\pi_{2m}(\Sigma H_{n, \ell}))\geq \frac{1}{(2m)^{1+\varepsilon}}\Big(\frac{3+\sqrt{5}}{2}\Big)^{\frac{m}{2(n+\ell)-1}}.
\qedhere\popQED 
\]
\end{corollary}

Consider the $n$-th unitary group $U(n)$ which is connected and of real dimension $n^2$.
In \cite[Example 2.8]{Boy1}, it is showed that when $n\geq 3$ there is a map 
\[
S^3\vee S^5 \stackrel{}{\longrightarrow}U(n)
\]
which induces a surjection on $\widetilde{K}^\ast(~)\otimes \mathbb{Z}/p$ for all odd primes $p$.
It is clear that this map can be lifted to the special unitary group $SU(n)$, which is $2$-connected and of real dimension $n^2-1$.
Therefore, the following corollary follows immediately from Theorem \ref{secondKTheoremw}, which strengthens \cite[Example 2.8]{Boy1}.
\begin{corollary}\label{Kcoro3intro}
Let $p$ be an odd prime and $n\geq3$.
Then for any $\varepsilon >0$, once $m$ is large enough we have
\[
\begin{split}
\mathrm{rank}_{p}(\pi_{m}(\Sigma U(n)))&\geq \frac{1}{m^{1+\varepsilon}}\varphi^{\frac{m}{n^2+1}}>\frac{1}{m^{1+\varepsilon}} (1.19)^{\frac{m}{n^2+1}}, \\
\mathrm{rank}_{p}(\pi_{m}(\Sigma SU(n)))&\geq \frac{1}{m^{1+\varepsilon}}\varphi^{\frac{3m}{n^2}}>\frac{1}{m^{1+\varepsilon}} (1.70)^{\frac{m}{n^2}},
\end{split}
\]
where $\varphi$ is the unique positive real root of $z^5-z^2-1=0$. ~$\qqed$
\end{corollary}

The structure of this paper is as follows. Section \ref{algebraSection} treats the algebra and combinatorics. Subsection \ref{noDiffSubsection} treats free Lie algebras without a differential, and Subsection \ref{diffSubsection} studies the module of boundaries in the differential case. These results are then used to prove the main theorems in Section \ref{topologySection}.

\bigskip

\noindent{\bf Acknowledgements} 
This paper was written while Guy Boyde was an EPSRC Doctoral Prize postdoc at the University of Southampton. He would like to thank Naomi Andrew, George Davenport, Lawk Mineh, and Stephen Theriault for many helpful conversations.

Ruizhi Huang was supported in part by the National Natural Science Foundation of China (Grant nos. 11801544 and 12288201), the National Key R\&D Program of China (No. 2021YFA1002300), the Youth Innovation Promotion Association of Chinese Academy Sciences, and the ``Chen Jingrun'' Future Star Program of AMSS.

%%%%%%%%%%%%%%%%%%%%%%%%%%%%%%%%%%%%%%%%%%%%%%%%%%%%%%%%%%%%%%%%%%%%%%

\section{Algebra} \label{algebraSection}

\subsection{Complex arithmetic}

\begin{lemma} \label{complexGCD} Let $S$ be a finite set of positive integers, let $g = \gcd(S)$, and let $\eta \in \mathbb{C}$ be nonzero. Then $\eta^g$ is a positive real if and only if $\eta^i$ is a positive real for all $i \in S$. \end{lemma}

\begin{proof} The `only if' direction follows from the fact that $g$ divides each member of $S$. For the `if' direction, Bezout's Lemma gives $\alpha_i \in \mathbb{Z}$ for each $i \in S$ such that $\sum_i \alpha_i \cdot i = g$. Thus, if each $\eta^i$ is a positive real, we get $$\eta^g = \eta^{\sum_i \alpha_i \cdot i} = \prod_i (\eta^i)^{\alpha_i},$$ which is a product of powers of positive reals, hence also a positive real. \end{proof}

\begin{lemma} \label{polynomialProperties} Let $c_0, \dots c_{k-1} \in \mathbb{Z}_{\geq 0}$, with $c_0 \geq 1$. The polynomial $$P(z) = z^{k} - \sum_{i=0}^{k-1} c_i z^i$$ has precisely one positive real root, $\varphi$, which occurs with multiplicity one, and satisfies $\varphi \geq 1$. The other roots $\eta$ satisfy $|\eta| \leq \varphi$, with equality holding if and only if $\eta$ is the product of $\varphi$ with a $g$-th root of unity, where $g=\gcd(\{i \mid c_i \neq 0\}\cup \{k\})$. \end{lemma}

\begin{proof} The number of sign changes between consecutive coefficients in $P$ is 1, so $P$ has precisely one positive real root by Descartes' rule of signs. Call this root $\varphi$. Rearranging, we have $\varphi^k = \sum_{i=0}^{k-1} c_i \varphi^i$. Since $c_0 \geq 1$, we must have $\varphi^k \geq 1$, so $\varphi \geq 1$.

Suppose that $\eta \in \mathbb{C}$ is a root of $P$. Taking modulus and applying the triangle inequality, we obtain $$| \eta | ^{k} = | \eta^{k} | = | \sum_{i=0}^{k-1} c_i \eta^i | \leq \sum_{i=0}^{k-1} c_i |\eta|^ i.$$ 

Equality holds in the above if and only if 1) $\eta^i$ is a non-negative real for all $i$ for which $c_i \neq 0$, and 2) $|\eta|$ is a root of $P$. By Lemma \ref{complexGCD}, the first condition is equivalent to $\eta^g$ being a non-negative real, where $g = \gcd \{i \ \lvert \ c_i \neq 0 \}$. The second condition is equivalent to $|\eta| = \varphi$, since $|\eta|$ is a non-negative real. The root $\varphi$ satisfies these conditions, and the other solutions are obtained as the product of $\varphi$ with the $g$-th roots of unity.

If the inequality is strict then we have $P(|\eta|) < 0$. Since the value of the polynomial $P(|z|)$ is positive for sufficiently large $|z|$, and $\varphi$ is the unique positive real root, $P(|z|)>0$ for any $|z|>\varphi$. It follows that $|\eta| < \varphi$, as required.
\end{proof}

For a polynomial $P(z)$ of degree $k$, let $\eta_1, \dots, \eta_k$ be the roots of $P$, with multiplicity. The $N$-th \emph{Newton polynomial} of $P$ is the complex number $\eta_1^N+ \dots + \eta_k^N$. If $P$ has real coefficients, then the roots occur in conjugate pairs and the Newton polynomials take real values. Lemma \ref{polynomialProperties} controls the Newton polynomials quite tightly. In particular, the next lemma explains that when $N=gn$ is $g$-divisible they are well approximated asymptotically by $g \varphi^{gn}$, and when $N$ is not $g$-divisible they are approximated by zero with the same error.

\begin{lemma} \label{NewtonGrowth} Let $c_0, \dots c_{k-1} \in \mathbb{Z}_{\geq 0}$, with $c_0 \geq 1$. As $n \to \infty$, the Newton polynomials of $$P(z) = z^{k} - \sum_{i=0}^{k-1} c_i z^i$$ satisfy
\begin{itemize}
    \item For $N$ not divisible by $g$ we have $$\lvert \eta_1^{N} + \dots + \eta_k^{N} \rvert \leq (k-g)|\psi|^{N}.$$
    \item When $N=gn$ is $g$-divisible we have $$\lvert g \varphi^{gn} - (\eta_1^{gn} + \dots + \eta_k^{gn}) \rvert \leq (k-g)|\psi|^{gn},$$
\end{itemize}
where $\varphi$ is the unique positive real root of $P(z)$, $\psi$ is the next largest root in absolute value, and $g=\gcd(\{i \mid c_i \neq 0\}\cup \{k\})$. \end{lemma}

By definition, $r_N$ is a sum of $N$-th powers of positive reals less than $\varphi$. This means that this lemma implies for example that $(\eta_1^{gn} + \dots + \eta_k^{gn}) \sim g \varphi^{gn}$ as $n \to \infty$.

\begin{proof} By Lemma \ref{polynomialProperties}, roots of $P(z)$ come in two kinds: those which are the product of $\varphi$ with a $g$-th root of unity, and those roots $\eta$ with $|\eta| < \varphi$ (hence $|\eta| \leq | \psi |$). The important point is that each root of the first kind occurs with multiplicity precisely 1.

To see this, apply Lemma \ref{polynomialProperties} to the polynomial $$P'(z) = z^{\frac{k}{g}} - \sum_{i=0}^{k-1} c_i z^{\frac{i}{g}}$$ obtained by dividing all powers by $g$, and use the fact that roots of $P(z)$ are precisely the $g$-th roots of the roots of $P'(z)$.

Then, without loss of generality assume $\eta_1$, \dots $\eta_g$ are the roots of the first kind, so that $\lvert \eta_1 \rvert = \dots = \lvert \eta_g \rvert = \varphi$. From elementary complex analysis or group theory we have that $\eta_1^N + \dots + \eta_g^N = \begin{cases} g \varphi^N & g \mid N, \\
0 & g \nmid N,
\end{cases}$ and the result then follows from the triangle inequality. \end{proof}

\subsection{Free Lie algebras} \label{noDiffSubsection}

We write the generating set $X$ of a free Lie algebra $L=L(X)$ over $\mathbb{Z}$ as follows. Write $q_1 < \dots < q_\ell$ for the distinct degrees which contain an element of $X$. Write $x_{i,1}, x_{i,2}, \dots, x_{i,m_i}$ for the distinct generators in degree $q_i$, so that in particular the number of generators in degree $q_i$ is $m_i$. Hilton \cite{Hil} showed that $L$ is free as a $\mathbb{Z}$-module. 

Let $\mu: \mathbb{Z}_{> 0} \longrightarrow \{-1, 0, 1\}$ be the \emph{M\"obius inversion function}, defined by $$\mu(s) = \begin{cases} 
      1 & s=1 \\
      0 & s>1 \textrm{ is not square free} \\
      (-1)^{\ell} & s>1 \textrm{ is a product of $\ell$ distinct primes.}
      \end{cases}$$
      
Given a polynomial $P(z) = a_0 + a_1 z + \dots + a_k z^k$ with $ a_0\neq 0$, the \emph{reciprocal} of $P(z)$ is $a_k + a_{k-1} z + \dots + a_0 z^k$. For given $P(z)$, let $\eta_1, \dots \eta_k$ be the complex roots of the reciprocal of $P(z)$, with multiplicity (so $P(z)=a_0 \prod_{i=1}^k (1-\eta_i z)$). Write $$S_{N}(P(z)):=\eta_1^N + \dots + \eta_k^N$$ for the $N$-th Newton polynomial in the zeroes of the reciprocal.
      
The following theorem is due to Babenko. Relative to his statement, we have changed variable using the fact that, for fixed $N$, $d \mapsto \frac{N}{d}$ is a self-bijection of the set of divisors of $N$.

\begin{theorem}{\cite[Proposition 1]{Bab}} \label{BabenkoFormula} Let $L$ be the free graded Lie algebra over $\mathbb{Z}$ on a finite set of generators $\{x_{i,j}\}$, with notation as above. Then $${\rm rank}(L_N)=\frac{(-1)^N}{N} \sum_{d \mid N} (-1)^{\frac{N}{d}} \mu(d) S_{\frac{N}{d}}(1- \sum_{i=1}^{\ell} m_i z^{q_i}),$$
where the sum is taken over the divisors $d$ of $N$. \qed \end{theorem}

Our next theorem is essentially a result of Lambrechts \cite[Proposition 1]{Lam} in the special case of free Lie algebras. Our derivation of this result from Babenko's is essentially the same as Lambrechts's, but the situation is simpler and slightly more is true. The point of the theorem is that when $g \mid N$, ${\rm rank}(L_N)$ is well-approximated by $\frac{g}{N} \varphi^{N}$ with an error term given by a sum of exponentials in smaller bases.

\begin{theorem} \label{withoutDifferentials} Let $L$ be the free graded Lie algebra over $\mathbb{Z}$ on a finite set of generators $X$. As before, write $q_1 < \dots < q_\ell$ for the distinct degrees which contain an element of $X$, and let $g= \gcd(q_i)$. Let $m_i$ be the number of generators in degree $q_i$. \begin{itemize}
    \item If $g \nmid N$, then ${\rm rank}(L_N) = 0$.
    \item If $g \mid N$, then $\lvert {\rm rank}(L_N) - \frac{g}{N} \varphi^{N} \rvert \leq \frac{q_{\ell}}{N}|\psi|^{N}+g \varphi^{\frac{N}{2}}+q_{\ell}|\psi|^{\frac{N}{2}}$,
\end{itemize} where $\varphi$ is the unique positive real root of the degree $q_\ell$ polynomial $$P(z)=z^{q_{\ell}} - \sum_{i=1}^\ell m_i z^{q_\ell - q_i} = 0,$$ and $\psi$ is the next largest root in absolute value. In particular, $\varphi \geq (\sum_{i=0}^\ell m_i)^{\frac{1}{q_\ell}} = (\sum_{i=0}^\ell m_i)^{\frac{1}{\max(q_1, \dots q_\ell)}}$. \end{theorem}

If $P(z)$ has no roots which are strictly smaller than $\varphi$ in absolute value (i.e. `$\psi$ does not exist') then terms involving $\psi$ may be disregarded: precisely, the inequality in the second bullet may be replaced by $\lvert {\rm rank}(L_N) - \frac{g}{N} \varphi^{N} \rvert \leq g \varphi^{\frac{N}{2}}$.

\begin{proof} The first bullet follows immediately from the fact that $L$ is concentrated in degrees divisible by $g$.

We will now prove the second bullet. The point is that the Babenko's formula of Theorem \ref{BabenkoFormula} is dominated by the $d=1$ term. Let $N$ be divisible by $g$. By Theorem \ref{BabenkoFormula} (using that $\mu(1)=1$) we have

\begin{equation*}
\begin{split}
{\rm rank}(L_N) &=\frac{(-1)^N}{N} \sum_{d \mid N} (-1)^{\frac{N}{d}} \mu(d) S_{\frac{N}{d}}(1- \sum_{i=1}^{\ell} m_i z^{q_i}) \\
&= \frac{1}{N}S_{N}(1- \sum_{i=1}^{\ell} m_i z^{q_i}) + \frac{(-1)^N}{N} \sum_{\substack{d \mid N \\ d \geq 2}} (-1)^{\frac{N}{d}} \mu(d) S_{\frac{N}{d}}(1- \sum_{i=1}^{\ell} m_i z^{q_i}).
\end{split}
\end{equation*}

We name these two terms, writing $S_N=S_{N}(1- \sum_{i=1}^{\ell} m_i z^{q_i})$ to simplify notation. Let $$A_N := \frac{1}{N}S_{N},$$ and let $$B_N := \frac{(-1)^N}{N} \sum_{\substack{d \mid N \\ d \geq 2}} (-1)^{\frac{N}{d}} \mu(d) S_{\frac{N}{d}}.$$

By Lemma \ref{NewtonGrowth} (with $n = \frac{N}{g}$), we have $\lvert S_N - g \varphi^{N} \rvert \leq (q_\ell - g) |\psi|^N \leq q_\ell |\psi|^N$ for $\varphi$ and $\psi$ as in the theorem statement. It therefore suffices to show that $\lvert B_N \rvert \leq g \varphi^{\frac{N}{2}} + q_{\ell}|\psi|^{\frac{N}{2}}.$

Since $\lvert \mu(d) \rvert \leq 1$, we have by Lemma \ref{NewtonGrowth} that

$$ \lvert B_N \rvert = \frac{1}{N} \lvert  \sum_{\substack{d \mid N \\ d \geq 2}} (-1)^{\frac{N}{d}} \mu(d) S_{\frac{N}{d}} \rvert \\
 \leq \frac{1}{N} \sum_{\substack{d \mid N \\ d \geq 2}}  \lvert S_{\frac{N}{d}} \rvert \leq \frac{1}{N} \sum_{\substack{d \mid N \\ d \geq 2}}( g \varphi^{\frac{N}{d}} + q_{\ell}|\psi|^{\frac{N}{d}}).$$
 
 The number of terms in this summation is at most the number of divisors of $N$, which is at most $N$. The term is a sum of exponentials in positive bases, hence is strictly increasing, and in particular for $d \geq 2$ we have the termwise bound $g \varphi^{\frac{N}{d}} + q_{\ell}|\psi|^{\frac{N}{d}} \leq g \varphi^{\frac{N}{2}} + q_{\ell}|\psi|^{\frac{N}{2}}$. Putting this together gives
 
 $$\lvert B_N \rvert \leq  g \varphi^{\frac{N}{2}} + q_{\ell}|\psi|^{\frac{N}{2}},$$ as required.

Lastly, we check that $\varphi \geq (\sum_{i=0}^\ell m_i)^{\frac{1}{q_\ell}}$. Since the polynomial $P(z) = z^{q_{\ell}} - \sum_{i=1}^\ell m_i z^{q_\ell - q_i}$ has a unique positive root by Lemma \ref{polynomialProperties}, it suffices to check that $P((\sum_{i=0}^\ell m_i)^{\frac{1}{q_\ell}})$ is non-positive. For each $i$, $\frac{q_\ell - q_i}{q_\ell}$ lies between 1 and 0, so for any $x \geq 1$ we have $x^{\frac{q_\ell - q_i}{q_\ell}} \geq 1$. It follows that $$P((\sum_{i=0}^\ell m_i)^{\frac{1}{q_\ell}}) = (\sum_{i=0}^\ell m_i) - \sum_{i=1}^\ell m_i (\sum_{i=0}^\ell m_i)^{\frac{q_\ell - q_i}{q_\ell}} \leq (\sum_{i=0}^\ell m_i) - \sum_{i=1}^\ell m_i \cdot 1 = 0,$$ as required. \end{proof}

\subsection{Free Lie algebras with differentials} \label{diffSubsection}

Free Lie algebras over $\Zpr$ are obtained by tensoring the corresponding free Lie algebra over $\mathbb{Z}$ with $\Zpr$, since this gives the correct universal property.

In this subsection, we consider $L=L(x,y)=L(x,dx)$, the free differential Lie algebra over $\Zpr$ on the acyclic rank 2 free differential $\Zpr$-module on generators $x$ and $y$ ($dx = y$). Suppose that $\deg(x)=q+1$, so $\deg(y)=q$. By Theorem \ref{withoutDifferentials}, since $\gcd(q,q+1)=1$, we know that $${\rm rank}_{\Zpr}(L_N) \sim \frac{1}{N} \varphi^{N},$$ where $\varphi$ is the unique positive real root of the degree $q+1$ polynomial $$z^{q+1}-z-1 = 0.$$

The size of the error in this approximation is exponential in base depending on the next largest root $\psi$ (in absolute value), and $\sqrt{\varphi}$.

In this subsection we are instead interested in $B := \textrm{Im}(d) \subset L$, the module of boundaries. Our aim is to prove Theorem \ref{withDifferentials}. The argument will go as follows. It is known (Theorem \ref{acyclicHomology}) that the differential on $L$ is `almost acyclic'. A counting argument using the fact that ${\rm rank}(L_N) \sim \frac{1}{N} \varphi^{N}$ then shows that the rank of the module of boundaries must be asymptotically a fixed fraction of that of $L_N$.

We will first reduce to the case $r=1$ by means of the following lemma, which is proven in \cite{Boy2} as Lemma 7.10. 

\begin{lemma} \label{reduceToField} Let $\varphi: M \longrightarrow N$ be a map of $\Zpr$-modules, with $N$ free. Then ${\rm rank}_{\Zpr}(\mathrm{Im}(\varphi)) = {\rm rank}_{\Zp}(\mathrm{Im}(\varphi \otimes \Zp))$. \qed \end{lemma}

Now assume $r=1$. Let $u$ be an even-dimensional class in a graded differential Lie algebra $L$ over $\Zp$ for $p \neq 2$. Following \cite{CMN}, let $$\tau_k(u) = \mathrm{ad}^{p^k-1}(u)(du),$$ and let $$\sigma_k(u) = \frac{1}{2} \sum_{j=1}^{p^k-1} \frac{1}{p}{p^k \choose j}[\mathrm{ad}^{j-1}(u)(du), \mathrm{ad}^{p^k-1-j}(u)(du)].$$

From our point of view, the point of the next theorem is that free differential Lie algebras are almost acyclic.

\begin{theorem}{\cite[Proposition 4.9]{CMN}} \label{acyclicHomology} Let $V$ be an acyclic differential $\Zp$-vector space. Write $L(V) \cong HL(V) \oplus K$, for an acyclic module $K$. If $K$ has an acyclic basis, that is, a basis $$\{x_\alpha, y_\alpha, z_\beta, w_\beta\},$$ where $\alpha$ and $\beta$ range over index sets $\mathscr{I}$ and $\mathscr{J}$ respectively, and we have $$d(x_\alpha) = y_\alpha, \mathrm{ deg}(x_\alpha) \textrm{ even,}$$ $$d(z_\beta) = w_\beta, \mathrm{ deg}(z_\beta) \textrm{ odd,}$$ then $HL(V)$ has a basis \[ \pushQED{\qed} 
\{\tau_k(x_\alpha), \sigma_k(x_\alpha)\}_{\alpha \in \mathscr{I}, k \geq 1} .\qedhere
\popQED \] \end{theorem}

The theorem implies that the differential on $L$ can be modified slightly to make it acyclic. Namely, define a new differential $\overline{d} : L(V) \to L(V)$ by setting $\overline{d} = d$ on $K$, and letting  $\overline{d}(\tau_k(x_\alpha)) = \sigma_k(x_\alpha)$, $\overline{d}(\sigma_k(x_\alpha)) = 0$. Of course, $\overline{d}$ will no longer satisfy the Leibniz rule, but it will still be a vector space endomorphism of degree $-1$ which satisfies $\overline{d}^2=0$.

Now let $\overline{B} := \mathrm{Im}(\overline{d}) \subset L$, and let $\sigma \subset L$ be the subspace spanned by the elements $\sigma_k(x)$, for some even degree $x \in L$ and $k \in \mathbb{Z}^+$. By definition of $\overline{d}$ we then have the following corollary.

\begin{corollary} \label{differentialCorrection} We have $\overline{B}_N \cong B_N \oplus \sigma_N$. \qed \end{corollary}

The next lemma justifies the approximation by providing a crude upper bound on $\sigma_N$.

\begin{lemma} \label{sigmaIsSmall} We have the bound $$\dim_{\Zp} \sigma_N \leq c_1 \cdot N \varphi^{\frac{N}{p}},$$ where $c_1 = 2(q+2)\varphi^{\frac{2}{p}}$. \end{lemma}

\begin{proof} By definition, $\sigma_N$ is spanned by classes $\sigma_k(x_\alpha)$, and we have $\deg(\sigma_k(x_\alpha)) = k \deg(x_\alpha)-2$. We therefore have $$\dim_{\Zpr} \sigma_N \leq \sum_{\substack{M \leq N \\ p^k M - 2 =N}} \dim_{\Zp} L_M \leq \sum_{\substack{M \leq N \\ p^k M - 2 =N}} (\frac{1}{M} \varphi^M + \frac{q+1}{M} |\psi|^{M} +  \varphi^{\frac{M}{2}} + (q+1)|\psi|^{\frac{M}{2}})$$
$$\leq \sum_{\substack{M \leq N \\ p^k M - 2 =N}} ( (q+2) \varphi^{M} +(q+2)\varphi^{\frac{M}{2}}) \leq \sum_{\substack{M \leq N \\ p^k M - 2 =N}} 2(q+2) \varphi^{M}.$$
by Theorem \ref{withoutDifferentials} (we use $|\psi| < \varphi$, and then drop the factors of $\frac{1}{M}$, to obtain a bound which is strictly increasing even for small $M$). This summation contains fewer than $N$ terms, and since the value of a given term is increasing in $M$, the size of the largest term is controlled by $M = \frac{N+2}{p^k} \leq \frac{N+2}{p}$, so $$\dim_{\Zp} \sigma_N \leq N \cdot 2(q+2) \varphi^{\frac{N+2}{p}},$$ as required. \end{proof}

We next estimate the size of $\dim \overline{B}_N$.

\begin{lemma} \label{BBarIsBig} Let $\psi$ be the next largest (in absolute value) root of $z^{q+1}-z-1$ after $\varphi$. We have $$\dim_{\Zp} \overline{B}_N \geq  (1-\frac{N}{N-1} \frac{1}{\varphi}) \cdot \frac{1}{N} \varphi^N - \kappa \lvert \psi \rvert^N - c_2 \varphi^{\frac{N}{2}},$$ where $\kappa = (q+1)(1+ \frac{1}{\lvert \psi \rvert})$ and $c_2 = (q+2)(1+ \frac{1}{\sqrt{\varphi}}) \leq 2(q+2)$. \end{lemma}

\begin{proof} Since $\overline{d}$ is acyclic, we have $\overline{B}_N = \mathrm{Ker}(\overline{d}: L_N \to L_{N-1})$. The First Isomorphism Theorem then gives that $\faktor{L_N}{\overline{B}_N} \cong \overline{B}_{N-1}$, and since $\overline{B}_{N-1} \subset L_{N-1}$ we get $$\dim_{\Zp} \overline{B}_N \geq \dim_{\Zp} L_N - \dim_{\Zp}L_{N-1}.$$

Theorem \ref{withoutDifferentials} gives (since $g=1$ and $|\psi| < \varphi$) $$\dim_{\Zp}L_{N-1} \leq \frac{1}{N-1}\varphi^{N-1} + \frac{q+1}{N-1} \lvert \psi \rvert^{N-1} + (q+2) \varphi^{\frac{N-1}{2}},$$ and $$\dim_{\Zp}L_{N} \geq \frac{1}{N}\varphi^{N} - \frac{q+1}{N} \lvert \psi \rvert^{N} - (q+2)\varphi^{\frac{N}{2}}.$$

Combining these inequalities gives the result. \end{proof}

We are now ready to state and prove the main theorem of this subsection.

\begin{theorem} \label{withDifferentials} Let $L \otimes \Zpr=L(x,d x) \otimes \Zpr = L(x,y) \otimes \Zpr$ be the free differential graded Lie algebra over $\Zpr$ on two generators $x$ and $y$ satisfying $y= d x$. Let $q = \deg (y)$, so that $\deg (x) = q+1$. Let $B = \textrm{Im} (d) \subset L \otimes \Zpr$ be the submodule of boundaries. Then we have the bound $$ {\rm rank}(B_N) \geq (1 - \frac{N}{N-1} \frac{1}{\varphi}) \cdot \frac{1}{N} \varphi^{N} - c N \varphi^{\frac{N}{2}} -  \kappa \lvert \psi \rvert^N,$$ where $\varphi > 1$ is the unique positive real root of the degree $n$ polynomial $$z^{q+1} - z - 1 = 0,$$ $\psi$ is the next largest root in absolute value, $c= 2(q+2)(1+ \varphi)$, and $\kappa = (q+1)(1+ \frac{1}{\lvert \psi \rvert})$. We have the bounds $2^{\frac{1}{q+1}}< \varphi < 1 + \frac{1}{q}$. \end{theorem}

\begin{proof} By Lemma \ref{reduceToField}, it suffices to prove the theorem in the case $r=1$. By Lemma \ref{differentialCorrection} we have $$\dim_{\Zp} B_N \geq \dim_{\Zp} \overline{B}_N - \dim_{\Zp} \sigma_N.$$ Combining Lemmas \ref{sigmaIsSmall} and \ref{BBarIsBig} ($c=c_1+c_2$) then gives the result. \end{proof}

\section{Topology} \label{topologySection}

\subsection{Homology}

We now prove Theorem \ref{homologyResultNoSuspension}.

\begin{proof}[Proof of Theorem \ref{homologyResultNoSuspension}] In the proof of Theorem 1.5 of \cite{Boy2} it is shown that there exists a commutative diagram (the details of the definitions of the maps need not concern us here):

\begin{center}
\begin{tabular}{c}
\xymatrix{
L'(x,y) \ar^{\theta \circ d}[d] \ar^{\beta^r \circ \Phi_\pi^{r,r}}[r] & \pi_*(\Omega P^{n+1}(p^r)) \ar^(0.57){(\Omega \mu)_*}[r] \ar^{h \circ \rho^s}[d] & \pi_*(\Omega Y) \ar^{h \circ \rho^s}[d] \\
L(x,y) \otimes \Zps \ar^{\Phi_H^{r,s}}[r] & H_*(\Omega P^{n+1}(p^r); \Zps) \ar^(0.57){(\Omega \mu)_*}[r] & H_*(\Omega Y; \Zps).
}
\end{tabular}
\end{center}

In the diagram, $L(x,dx) \otimes \Zps$ is the free differential Lie algebra (with $dx=y$ , $\deg(x) = q+1$, $\deg(y) = q$). The top-left entry $L'(x,dx)$ is a certain module over $\Zpr$ which is `almost' a free differential Lie algebra.

We now use various results from \cite{Boy2}. By the remark immediately before Corollary 8.9 of that paper, the image of the left hand vertical map $\theta \circ d$ is precisely the module of boundaries $BL$. By Lemma 9.6 the map $\Phi_H^{r,s}$ is an injection, and the induced map on homology, $(\Omega \mu)_*$, is an injection by assumption. It follows by commutativity that the image in the bottom-right, $I := \mathrm{Im}(h \circ \rho^s \circ (\Omega \mu)_* \circ \beta^r \circ \Phi_\pi^{r,r})$, is isomorphic to $BL$.

The point is then that the homotopy groups of $Y$ surject onto $I$, hence must be just as large. More precisely, we obtain that $$\sum_{t=s}^r {\rm rank}_{\Zpt}(\pi_N(\Omega Y)) \geq {\rm rank}_{\Zps}(I_N) =  {\rm rank}_{\Zps}(BL_N)$$ by Lemma 7.8 of \cite{Boy2} applied to the part of the diagram consisting of

\begin{center}
\begin{tabular}{c}
\xymatrix{
\pi_N(\Omega Y) \ar[dr] & \\
(L'(x,y))_N \ar[r] \ar[u] & H_N(\Omega Y; \Zps).
}
\end{tabular}
\end{center}

The loops on $Y$ is just a degree shift on homotopy groups, so the result follows by Theorem \ref{withDifferentials} of this paper.
\end{proof}

\subsection{$K$-theory}

In this subsection, the following linear inequality relating integers $j$ and $N$ will arise often. We will refer to it as Condition $(\ref{eqn})$. Here, $X$ is a fixed space, $\mathrm{conn}(X)$ is the $p$-local connectivity of $X$, and $\dim(X)$ is the largest $d$ for which $H^d(X; \mathbb{Q}) \neq 0$.

\begin{equation} j > \frac{1}{2(p-1)}(\frac{\dim(X)+1}{\mathrm{conn}(X)+1}-1)N+(\frac{\dim(X)+1}{\mathrm{conn}(X)+1})(\mathrm{conn}(X)+2)-1), \label{eqn} \tag{*} \end{equation}

The next theorem refines and slightly generalises Theorem 1.4 of \cite{Boy1}.

\begin{theorem} \label{basicKTheorem} Let $p$ be an odd prime, and let $X$ be a path connected space having the $p$-local homotopy type of a finite $CW$-complex. Suppose that there exists a map $$\mu : \bigvee_{i=1}^\ell \bigvee_{j=1}^{m_i} S^{q_i+1} \to \Sigma X$$
with $1 \leq q_1 < q_2 < \dots < q_\ell$, such that the map $$\widetilde{K}^*(\Sigma X) \otimes \Zp \xrightarrow{\mu^*} \widetilde{K}^*(\bigvee_{i=1}^\ell \bigvee_{j=1}^{m_i} S^{q_i+1}) \otimes \Zp \cong \bigoplus_{i=1}^\ell \bigoplus_{j=1}^{m_i} \Zp$$
is a surjection.

Then for any $N$, $j$ such that $$j > \frac{1}{2(p-1)}(\frac{\dim(X)+1}{\mathrm{conn}(X)+1}-1)N+(\frac{\dim(X)+1}{\mathrm{conn}(X)+1})(\mathrm{conn}(X)+2)+1)$$ (i.e. such that Condition (\ref{eqn}) holds) we have $$\sum_{t=1}^\infty \mathrm{rank}_{\Zpt}(\pi_{N+2j(p-1)-1}(\Omega \Sigma X)) \geq {\rm rank}_{\Zpr}(L_N \otimes \Zpr),$$ where $L$ is as in Theorem \ref{withoutDifferentials} (the free Lie algebra on generators corresponding to the spheres in the wedge), $\mathrm{conn}(X)$ is the $p$-local connectivity of $X$, $\dim(X)$ is the dimension of $X$ as measured by rational cohomology, and $g=\gcd(q_1, \dots , q_\ell)$.  \end{theorem}

\begin{proof} This is essentially a more careful restatement of Theorem 1.4 of \cite{Boy1}. Some of the arguments of that paper are given only for a wedge of two spheres, but all of them apply verbatim to any finite wedge. Construction 7.15 of that paper gives (in slightly different language) a diagram of the form
\begin{center}
\begin{tabular}{c}
\xymatrix{
\pi_N(\Omega \Sigma X) \ar[dr] & \\
L_N \otimes \Zpr \ar[r] \ar[u] & E_{N+2j(p-1)}
}
\end{tabular}
\end{center} for some module $E_*$ whose definition need not concern us.

Theorem 7.16 of that paper than says that the horizontal map is an injection, and hence, just as in the proof of Theorem \ref{homologyResultNoSuspension}, the conclusion holds, provided that there exists some $\ell \in \mathbb{Z}_{\geq 0}$ such that $\ell^{j(p-1)+\frac{N-1}{2}} > \lambda_\ell^k$, for an integer $k$ which may be taken to be $\lceil \frac{N+1}{\mathrm{conn}(X)+1} \rceil.$

The inequality therefore rearranges to $j > \frac{1}{p-1}(\lceil \frac{N+1}{\mathrm{conn}(X)+1} \rceil \frac{\log(\lambda_\ell)}{\log(\ell)}- \frac{N-1}{2}).$ In \cite{AA}, it is shown that $\lambda_\ell = \ell^{\lceil \frac{\dim(X)}{2} \rceil}$, so we may simplify to $$ j > \frac{1}{p-1}(\lceil \frac{N+1}{\mathrm{conn}(X)+1} \rceil \lceil \frac{\dim(X)}{2} \rceil - \frac{N-1}{2}), $$ which is implied by Condition (\ref{eqn}), using the fact that for an integer $z$ we have $\lceil \frac{z}{2}\rceil \leq \frac{z+1}{2}$. This completes the proof. \end{proof}

The next step is a simple application of Bezout's Lemma.

\begin{lemma} \label{BigBezout} Let $\alpha, \beta \in \mathbb{Z}$ with $\alpha, \beta  > 0$, and let $a,b \in \mathbb{R}$ with $a>0$. Consider the set of linear combinations $$S_n=\{n \alpha + j \beta \ \mid \ j \in \mathbb{Z}_{\geq 0}, \ j > an+b \} \subset \mathbb{Z}.$$ Let $g' = \gcd(\alpha, \beta )$. There exists a constant $B$, independent of $n$, such that for each $n$, all multiples of $g'$ which are at least $\min(S_n)+B$ are contained in $S_i$ for some $i$ which is close to $n$ in the sense that $n \leq i < n+\beta(\beta + 1)$. Furthermore, there exists a suitable $B$ satisfying the bound $B \leq \beta^2 (\alpha + a (1 + \beta))+ \beta$, and hence any $B \geq \beta^2 (\alpha + a (1 + \beta)) + \beta$ is also suitable. \end{lemma}

If $\alpha$ and $\beta$ are fixed (and $j$ and $n$ are allowed to vary) then it is a familiar fact that the set of integers of the form $n \alpha + j \beta$ is precisely the multiples of $g'$. Our statement is essentially just a more complicated version of this.

\begin{proof} First consider the set $S_n$. If an integer $j$ satisfies $j > an+b$ (so that $n \alpha + j \beta$ lies in $S_n$), then increasing the parameter $j$ certainly does not violate this condition. Therefore, adding a positive multiple of $\beta$ to an element of $S_n$ yields another element of $S_n$. In particular, $S_n$ already contains all integers which are obtained by increasing $\min(S_n)$ by a multiple of $\beta$. These values are by construction linear combinations of $\alpha$ and $\beta$, so they are all multiples of $g'$.

It remains, then, to show that by increasing $n$ `just a little', we can `fill in' the intermediate multiples of $g'$. We will do so by `giving ourselves enough room', in the sense of an ad-hoc quantity which we now define. Define the \emph{excess} of $(j,n)$ to be $j - (an+b)$. The condition $j > an+b$ is then equivalent to $(j,n)$ having positive excess.

By Bezout's Lemma, let $x > 0 $ and $y \geq 0$ be the solution of $x \alpha - y \beta = g'$ with smallest non-negative $y$. We have $0 < x \leq \frac{g'}{\alpha}+ \beta$ and $0 \leq y \leq \alpha$. Given an expression $n \alpha + j \beta$, replacing $n$ by $n+x$ and $j$ by $j-y$ increases the value of the linear combination $n \alpha + j \beta$ by $g'$, and reduces the excess by the constant $ax+y$. We will use this to fill in the remaining multiples of $g'$.

Let $j_0$ realise the smallest member of $S_n$, in the sense that $\min(S_n) = n \alpha + j_0 \beta$. Now take any $j \geq j_0 + \frac{\beta}{g'}(ax+y)$. The excess of $(n,j_0)$ was positive, so the excess of $(n,j)$ is greater than $\frac{\beta}{g'}(ax+y)$. We may therefore add $(x,-y)$ to $(n,j)$ up to $\frac{\beta}{g'}$ times while retaining a positive excess (and keeping $j$ non-negative). This shows that all multiples of $g'$ lying between $n \alpha + j \beta$ and $n \alpha + (j+1) \beta$ are contained in $S_i$ for some $i$ satisfying $n \leq i < n+ \frac{\beta}{g'} x$, and we may perform this procedure for any $j \geq j_0 + \frac{\beta}{g'}(ax+y)$. In particular, all multiples of $g'$ which are at least $\min(S_n)+\beta(\frac{\beta}{g'}(ax+y)+1)$ are contained in $S_i$ for some $i$ satisfying $n \leq i < n+ \frac{\beta}{g'} x$. The extra $+1$ here is because $j$ must be an integer. This is essentially the result, and it remains only to establish that we may take the constants as in the statement.

Now, $\frac{g'}{\alpha} \leq 1$, so $x \leq 1 + \beta$, and $\frac{\beta}{g'}x \leq \beta x \leq \beta (1+ \beta)$. This establishes the bounds on $i$. The bound on $B$ follows from these inequalities, together with $y \leq \alpha$. This completes the proof. \end{proof}

We now prove the following strong version of Theorem \ref{secondKTheoremw}.
\begin{theorem} \label{secondKTheorem} Let $p$ be an odd prime, and let $X$ be a path connected space having the $p$-local homotopy type of a finite $CW$-complex. Suppose that there exists a map $$\mu : \bigvee_{i=1}^\ell \bigvee_{j=1}^{m_i} S^{q_i+1} \to \Sigma X$$
with $1 \leq q_1 < q_2 < \dots < q_\ell$, such that the map $$\widetilde{K}^*(\Sigma X) \otimes \Zp \xrightarrow{\mu^*} \widetilde{K}^*(\bigvee_{i=1}^\ell \bigvee_{j=1}^{m_i} S^{q_i+1}) \otimes \Zp \cong \bigoplus_{i=1}^\ell \bigoplus_{j=1}^{m_i} \Zp$$
is a surjection.

Then there exist constants $\tau,\theta>0$ such that for multiples $M=mg'$ of $$g'=\gcd(q_1, \dots q_\ell,2(p-1))$$ we have $$\sum_{t=1}^\infty \mathrm{rank}_{\Zpt}(\pi_{M}(\Sigma X))\geq \frac{\tau}{\frac{1}{g}\frac{\mathrm{conn}(X)+1}{\dim(X)+1}M + \theta} \varphi^{(\frac{\mathrm{conn}(X)+1}{\dim(X)+1})M}-o(\frac{1}{M}\varphi^{(\frac{\mathrm{conn}(X)+1}{\dim(X)+1})M}),$$ where $\varphi$ is the unique positive real root of the degree $q_\ell$ polynomial $$z^{q_\ell} - \sum_{i=1}^\ell m_i z^{q_\ell - q_i} = 0,$$ (in particular, $\varphi \geq (\sum_{i=0}^\ell m_i)^{\frac{1}{q_\ell}} = (\sum_{i=0}^\ell m_i)^{\frac{1}{\max(q_1, \dots q_\ell)}}$), $\mathrm{conn}(X)$ is the $p$-local connectivity of $X$, $\dim(X)$ is the rational cohomological dimension of $X$, and $g=\gcd(q_1, \dots , q_\ell)$. \end{theorem}

\begin{proof}
Let $S_n$ be the set of dimensions $M$ for which Theorem \ref{basicKTheorem} tells us that $\sum_{t=1}^\infty \mathrm{rank}_{\Zpt}(\pi_{M}(\Sigma X)) \geq \dim_{\Zp}(L_{ng} \otimes \Zpr)$. That is: $$S_n = \{ng+j\cdot 2(p-1) \ \mid \ j \in \mathbb{Z}, \ j > an+b \} \subset \mathbb{Z},$$ where $a=\frac{g}{2(p-1)}(\frac{\dim(X)+1}{\mathrm{conn}(X)+1}-1)$ and $b = \frac{1}{2(p-1)}(\frac{\dim(X)+1}{\mathrm{conn}(X)+1}(\mathrm{conn}(X)+2)+1)$.

By Lemma \ref{BigBezout}, there exists a constant $B$, which may be taken to be $4(p-1)^2 (g + a (1 + 2(p-1)))+ 2(p-1)$ such that for each $M=mg' \geq \min(S_n) + B$, we have $\sum_{t=1}^\infty \mathrm{rank}_{\Zpt}(\pi_{M}(\Sigma X)) \geq \dim_{\Zp}(L_{ig} \otimes \Zpr)$ for some $i$ with $n \leq i < n + 8(p-1)^2$. By Theorem \ref{withoutDifferentials}, $$\sum_{t=1}^\infty \mathrm{rank}_{\Zpt}(\pi_{M}(\Sigma X)) \geq \frac{1}{i} \varphi^{ig} - q_{\ell}|\psi|^{ig}-g \varphi^{\frac{ig}{2}}-q_{\ell}|\psi|^{\frac{ig}{2}}.$$

Regardless of whether $|\psi| > 1$, we have $|\psi|^{\frac{ig}{2}} < 1 + |\psi|^{ig} < 2 + |\psi|^{(n+8(p-1)^2)g}$, so the inequality implies 
\begin{equation}\sum_{t=1}^\infty \mathrm{rank}_{\Zpt}(\pi_{M}(\Sigma X)) \geq \frac{1}{n+8(p-1)^2} \varphi^{ng}-g \varphi^{\frac{(n+8(p-1)^2)g}{2}} - q_{\ell}(3+2|\psi|^{(n+8(p-1)^2)g}).
    \label{ineq} \tag{$\dagger$}
\end{equation}
It remains only to find the dependency of $n$ upon $M$, and convert this expression into one in terms of $M$.

The smallest member of $S_n$ is obtained by taking the smallest $j=j_n$ satisfying Condition (\ref{eqn}). By definition $j_n$ is the smallest integer with $j_n > an+b$, so $j_n \leq an+b + 1$. Thus, $$\min(S_n)=ng+2j_n(p-1) \leq g(\frac{\dim(X)+1}{\mathrm{conn}(X)+1})n+2(p-1)(b+1).$$

To conclude, for given $M=mg'$, let $n=n(M)$ be the largest non-negative integer satisfying $M \geq g(\frac{\dim(X)+1}{\mathrm{conn}(X)+1})n+2(p-1)(b+1)+B$. Rearranging gives $n \leq \frac{M-(2(p-1)(b+1)+B)}{g} (\frac{\mathrm{conn}(X)+1}{\dim(X)+1})$. Since $n$ is the largest such integer, it is at least one less than this expression. Applying the bounds $n \leq i < n + 8(p-1)^2$, now gives that 
$$\frac{1}{n+8(p-1)^2} \varphi^{ng} \geq \frac{\tau}{\frac{1}{g}\frac{\mathrm{conn}(X)+1}{\dim(X)+1}M + \theta} \varphi^{\frac{\mathrm{conn}(X)+1}{\dim(X)+1}M},$$ for constants $\theta$ and $\tau$, and shows that the other terms of the inequality \ref{ineq} are $o(\frac{1}{M}\varphi^{\frac{\mathrm{conn}(X)+1}{\dim(X)+1}})$, as required. \end{proof}

\begin{remark} \label{Constants} 
In this remark we give the constants and error term for Theorem \ref{secondKTheorem}, and collect the other constants appearing in the proof.

The positive integers $q_1, \dots q_\ell$ are given in the hypotheses of Theorem \ref{secondKTheorem} (or Theorem \ref{withoutDifferentials}). Then $$g=\gcd(q_1, \dots , q_\ell),$$ and $$g' = \gcd(q_1, \dots q_\ell, 2(p-1)) = \gcd(g, 2(p-1)).$$

The space $X$ and the prime $p \neq 2$ are given in the hypotheses of Theorem \ref{secondKTheorem}, $\dim(X)$ is the rational dimension of $X$, and $\mathrm{conn}(X)$ is its $p$-local connectivity.

The constants appearing in the proof of Theorem \ref{secondKTheorem} are then
$$a=\frac{g}{2(p-1)}(\frac{\dim(X)+1}{\mathrm{conn}(X)+1}-1),$$  $$b = \frac{1}{2(p-1)}(\frac{\dim(X)+1}{\mathrm{conn}(X)+1}(\mathrm{conn}(X)+2)+1) \textrm{, and}$$ $$B=4(p-1)^2 (g + a (1 + 2(p-1))) + 2(p-1).$$

It then follows from the proof that the constants $\theta$ and $\tau$ of Theorem \ref{secondKTheorem} may be taken as follows.
$$\theta = 8(p-1)^2 - (\frac{\mathrm{conn}(X)+1}{\dim(X)+1})\frac{2(p-1)(b+1+B)}{g} \leq 8(p-1)^2 \textrm{, and}$$ $$\tau = \varphi^{-g-(\frac{\mathrm{conn}(X)+1}{\dim(X)+1})(2(p-1)(b+1)+B)},$$ where, as usual, $\varphi$ is the unique positive root of the polynomial $P(z) = z^{q_\ell} - \sum_{i=1}^\ell m_i z^{q_\ell - q_i} = 0.$ 

The error term in the bound Theorem \ref{secondKTheorem} is an unpleasant expression, and we restrict ourselves to noting that it is negative, and of the form $$-c_1 \varphi^{\frac{1}{2}\frac{\mathrm{conn}(X)+1}{\dim(X)+1} M} - c_2 |\psi|^{\frac{\mathrm{conn}(X)+1}{\dim(X)+1}M} - 3 q_\ell,$$ for constants $c_i$, where $\psi$ is the next largest root of $P(z)$ after $\varphi$, in absolute value. The deviation of the bound from being a pleasant expression is therefore exponential in bases determined by the roots of $P(z)$.
\end{remark}

%----------------------------------------------------------------------------------------------------------------------------------------------------------------------------------------------------------%

%%%%%%%%%%%%%%%%%%%%%%%%%%%%%%%%%%%%%%%%%%%%%%%%%%%%%%%%%%%%%%%%%%%%%%

%%% The bibliography %%%
\bibliographystyle{amsalpha}

\end{sloppypar}
\end{document}